\documentclass[12pt]{article}

\usepackage{amssymb}
\usepackage{amsthm}
\usepackage{amsmath}
\numberwithin{equation}{section}

\title{On the Representation of Holomorphic Functions  on Polyhedra}
\author{Jim Agler
\thanks{Partially supported by National Science Foundation Grant
DMS 1068830}
\and
John E. M\raise.5ex\hbox{c}Carthy
\thanks{Partially supported by National Science Foundation Grants DMS 0966845 and 1300280}
\and
N. J. Young
\thanks{Supported by London Mathematical Society grant 41219 and Engineering and Physical Sciences Research Council grant
EP/J004545/1}}

\renewcommand{\epsilon}{\varepsilon}

\def\norm#1{\| #1 \|}

\def\normdel#1{{\| #1 \|}_\delta}
\def\normmdel#1{\| #1 \|^-_\delta}
\def\normpdel#1{\| #1 \|^+_\delta}
\def\normpgam#1{\| #1 \|^+_\gamma}
\def\normm#1{{\| #1 \|}_m}
\def\normbm#1{{\| #1 \|}_{\bf m}}
\def\normtdel#1{{\| #1 \|}_{t\delta}}
\def\normptdel#1{\| #1 \|_{t\delta}^+}

\def\be{\begin{equation}}
\def\ee{\end{equation}}
\def\abs#1{\mid\!\! #1 \!\!\mid}
\def\set#1#2{\{ #1 \, | \, #2\}}

\def\cd{\mathbb{C}^d}
\def\cm{\mathbb{C}^m}
\def\dd{\mathbb{D}^d}
\def\dm{\mathbb{D}^m}
\def\cldm{{(\mathbb{D}^-)}^{m}}

\def\hinfofgd{H^\infty(G_\delta)}
\def\hinfdel{H_\delta^\infty}

\def\hinfm{H_{\bf m}^\infty}
\def\c{\mathcal{C}}
\def\h{\mathcal{H}}
\def\m{\mathcal{M}}

\def\lofh{\mathcal{L}(\mathcal{H})}
\def\hofg{\mathcal{H}(G)}

\def\rofg{\mathcal{R}(G)}

\def\pofg{\mathcal{P}(G)}

\def\gdel{G_\delta}

\def\kdel{K_\delta}
\def\gtdel{G_{t\delta}}

\def\cdel{{\mathcal{C}}_\delta}

\def\f{\mathcal{F}}
\def\fperp{\mathcal{F}^\perp}
\def\fdel{{\mathcal{F}}_\delta}
\def\fm{{\mathcal{F}}_m}
\def\ftdel{{\mathcal{F}}_{t\delta}}

\def\l{\lambda}

\begin{document}

\bibliographystyle{plain}

\maketitle
\begin{abstract}
We generalize the Oka extension theorem, and obtain bounds on the norm of the extension,
by using operator theory.
\end{abstract}

\newtheorem{defin}[equation]{Definition}
\newtheorem{lem}[equation]{Lemma}
\newtheorem{prop}[equation]{Proposition}
\newtheorem{thm}[equation]{Theorem}
\newtheorem{claim}[equation]{Claim}
\newtheorem{ques}[equation]{Question}

\section{Introduction}
\subsection{Oka's Theorem}
The following beautiful theorem of Oka, that gives a representation for holomorphic functions defined on $p$-polyhedra in $\cd$, has played a significant 
role in the development of several complex variables.
\begin{thm}\label{thm1.0}
(Oka \cite{oka36} as presented in \cite{alewer}) Let $\delta_1,\ldots,\delta_m$ be a collection of polynomials in $d$ variables normalized so that the p-polyhedron, $K_\delta$, defined by
$$K_\delta=\set{\lambda \in \cd}{\abs{\delta_l(\lambda)} \le 1 \text{ for }l=1,\ldots,m}$$ lies in $\dd$. If $\phi$ is holomorphic on a neighborhood of $K_\delta$, then there exists a function $\Phi$, holomorphic on a neighborhood of ${(\mathbb{D}^-)}^{d+m}$, such that $$
\phi(\lambda)=\Phi(\lambda,\delta(\lambda))
$$ for all $\lambda \in K_\delta$.
\end{thm}

Introduced originally in 1936 to give an elegant new proof of the Oka-Weil Approximation Theorem(\cite{wei35} and \cite{oka36}), Oka's theorem was a stem theorem for the development of the theory of analytic sheaves, a powerful tool for doing function theory on domains of holomorphy and more generally Stein spaces (\cite{gur} or \cite{hor}). Basic to the understanding of polynomial convexity,
 the theorem played an important role in the development of the theory of Banach Algebras. 
Many operator theorists first learn of Oka's Theorem in the context of one of its many
basic implications, the Arens-Calderon Trick (\cite{arc55}), which is fundamental to the spectral theory and the corresponding functional calculus for commuting tuples of operators (\cite{tay70a},\cite{tay70b},\cite{eschput}).

\subsection{Oka Mappings}
In this paper we shall show how ideas that are currently evolving within the operator theory community can be adapted to obtain precise bounds for Oka's Theorem. These bounds are defined using operator theory and the problem of computing them or indeed, even estimating them, in any meaningful fashion in terms of function theory remains in large part unexplored.

In addition to these new bounds that we will obtain, there is a second contribution presented in this paper to our understanding of Oka's Theorem. The idea is to drop Oka's normalization requirement and, more severely, not to allow the representing function $\Phi$ to ``see'' the coordinates $\lambda$.
Specifically, we introduce the following definition.
\begin{defin}
Let $\delta$ be an $m$-tuple of polynomials in $d$ variables. We say that $\delta$ is an Oka mapping if whenever $\phi$ is a function that is holomorphic on a neighborhood of $K_\delta$, there exists a function $\Phi$, holomorphic on a neighborhood of ${(\mathbb{D}^-)}^{m}$, such that $\phi =\Phi \circ \delta$ on $K_\delta$.
\end{defin}
Evidently, with this language, Oka's theorem becomes the assertion that if $\delta$ is an $m$-tuple of polynomials in $d$ variables and $K_\delta \subseteq \dd$, then $(\lambda,\delta)$ is an Oka mapping. Of course, this leaves open the question of whether or not the map $\delta$ itself is an Oka map.

One approach to the understanding of Oka mappings is  to use the Cartan Extension Theorem (\cite{car51}), and this provides
%
 a purely geometric characterization of Oka mappings. We let
$$G_\delta=\set{\lambda \in \cd}{\abs{\delta_l(\lambda)} < 1 \text{ for }l=1,\ldots,m}.$$ 
One always has
 $\gdel$ equal to the interior of $\kdel$, which we shall denote
by ${\kdel}^\circ$,  --- this is proved in  Lemma~\ref{lem2.1};
but it need not be the case that $\kdel = {\gdel}^-$, the closure of $\gdel$.
\begin{thm}\label{thm1.1}
If $\delta$ is an m-tuple of polynomials in $d$ variables, then $\delta$ is an Oka mapping if and only if
 there exists $t<1$ such that $\delta$ embeds $\gtdel$ as an analytic submanifold in $\frac{1}{t}\dm$ (i.e. $\delta$ is an injective, proper, unramified mapping from $\gtdel$ into $\frac{1}{t}\dm$).
\end{thm}

The implication $\Leftarrow$
follows from \cite[Thm. 7.1.5]{rud69}, and the converse from observing that, for some $t$, for each
coordinate function $\lambda^j$, $1 \leq j \leq d$, there is a function $\Phi_j$ holomorphic on $\frac{1}{t}\dm$
such that $\Phi_j (\delta (\lambda)) = \lambda^j$.

To return to operator theory,
we consider an analog of $\kdel$ but with points in $\cd$ replaced by $d$-tuples of pairwise commuting operators, $T=(T^1,T^2,\ldots,T^d)$, acting on a complex Hilbert space. Thus, we define
\be\label{eq1.0}
{\mathcal{F}}_\delta = \set{T}{\norm{{\delta}_l(T)} \le 1, l=1,\ldots,m}.
\ee
Already, with this one simple definition we obtain a second condition for $\delta$ to be an Oka map, now in operator theoretic terms.
\begin{prop}\label{thm1.2}
If $\delta$ is an m-tuple of polynomials in $d$ variables, then $\delta$ is an Oka mapping if and only if there exists $t<1$ such that $\ftdel$ is bounded.
\end{prop}
We remark that the classical Oka Theorem, Theorem \ref{thm1.0} above, is an immediate corollary of both Theorems \ref{thm1.1} and \ref{thm1.2}. This is due to the facts that for $t<1$, sufficiently close to $1$,  Oka's normalization condition, $K_\delta \subseteq \dd$ implies that $\gtdel 
\subseteq \dd$ and clearly, the map, $(t\lambda,t\delta)$, is  an analytic embedding of $\gtdel$ into $\frac{1}{t}{\mathbb D}^{m+d}$. Likewise, the family ${\mathcal{F}}_{(t\lambda,t\delta)}$ is bounded, as \eqref{eq1.0} implies that if $T \in {\mathcal{F}}_{(t\lambda,t\delta)}$, then for each $r=1,\ldots,d$, $\norm{T^r} \le \frac{1}{t}$.

Now, not all $\delta$ are Oka mappings. For an $m$-tuple of polynomials in $d$ variables, $\delta$, not necessarily assumed to be an Oka mapping, it is natural to ask the following question.
\begin{ques}
\label{qu1p6}
Given $\phi$ holomorphic on a neighborhood of $\kdel$, does there exist $\Phi$, holomorphic on  a neighborhood of ${(\mathbb{D}^-)}^{m}$, such that $\phi =\Phi \circ \delta$ on $K_\delta$?
\end{ques}

The approach to this question via the Cartan Extension Theorem would go something like this. First, we would hope that for $t<1$, sufficiently close to $1$, that $\delta(\gtdel)$ is an analytic variety in $\frac{1}{t}\dm$. Then the condition to  represent $\phi$ as in Oka's theorem
 would be that the function, $\delta^\sim$, defined on $\delta(\gtdel)$ by the formula,
$$\delta^\sim(\delta(\lambda)) = \phi(\lambda),$$
be a well defined analytic function on $\delta(\gtdel)$ that could be extended via the Cartan Theorem. The analysis of the analyticity of $\delta^\sim$ would require the full strength of analytic sheaf theory and would proceed with great difficulty. A fundamental problem with this approach, however, is that $\delta(\gtdel)$ need not be an analytic variety in $\frac{1}{t}\dm$ for any $t \leq1$.
For example, if $d=n=2, \ \delta_1=\lambda^1,\delta_2=\lambda^1\lambda^2$, then
\[
\delta (\gtdel) \ = \
\{ \lambda \, : \, |\l^1 | < \frac{1}{t},\, |\l^2 | < \frac{1}{t} \}
\setminus
\left[
\{ 0 \} \times \frac{1}{t} ( \dd \setminus \{ 0 \} ) \right]
.
\]

To answer the question in operator theory terms we return to $\fdel$ and notice that as a simple consequence of the Spectral Mapping Theorem we have that $\sigma(T) \subseteq K_\delta$ whenever $T \in \fdel$. Thus, if $\phi$ is holomorphic on a neighborhood of $\kdel$, then $\phi(T)$ can be defined by the Taylor functional calculus. Consequently, for $\phi$ holomorphic on a neighborhood of $\kdel$ we may define
\be \label{eq1.1}
{\normpdel{\phi}} \ = \ \sup_{T \in {\mathcal{F}}_\delta} \norm{\phi(T)}.
\ee
Question \ref{qu1p6} can be answered in terms of the quantity defined in \eqref{eq1.1}.
\begin{prop}\label{thm1.3}
If $\delta$ is an m-tuple of polynomials in $d$ variables and $\phi$ is holomorphic on a neighborhood of $\kdel$, then there exists $\Phi$ holomorphic on  a neighborhood of ${(\mathbb{D}^-)}^{m}$ such that $\phi =\Phi \circ \delta$ on $K_\delta$
 if and only if there exists $t<1$ such that $\normptdel{\phi} < \infty.$
\end{prop}
\subsection{Bounds for the Oka Representation}
\label{subsec1.3}

To describe our bounds for the Oka extension we shall employ a norm on holomorphic functions essentially introduced by von Neumann in \cite{vonN51}. This paper, which has had a profound influence on the development of operator theory, was the first to demonstrate that norms defined with the aid of operators can be natural from the point of view of function theory.
\begin{thm}\label{thm1.3a}
(von Neumann's Inequality \cite{vonN51}) If $C$ is a contraction acting on a complex Hilbert space and $\Phi$ is a function holomorphic on a neighborhood of $\mathbb{D}^-$, then
\be \label{eq1.2}
\norm{\Phi(C)} \le \max_{z \in \mathbb{D}^-}\abs{\Phi(z)}
\ee
\end{thm}
One can reformulate von Neumann's Inequality by saying 
 that if $\Phi$ is assumed holomorphic on a neighborhood of $\mathbb{D}^-$, then
\be\label{eq1.3}
\sup_{\norm{C} \le 1}\norm{\Phi(C)} = \max_{z \in \mathbb{D}^-}\abs{\Phi(z)}
\ee
Twelve years after von Neumann published his inequality, T. And\^o \cite{and63} 
proved a surprising and
subtle generalization to two variables.  If $\Phi$ is holomorphic on a
neighborhood of $({\mathbb{D}^-})^2$, then with $\mathcal{F}_2$ defined by
\be\label{eq1.4}
\mathcal{F}_2 = \set{C=(C_1,C_2)}{\norm{C_1}\le1,\norm{C_2}\le1,C_1C_2=C_2C_1},
\ee
the following analog of \eqref{eq1.3} obtains:
\be\label{eq1.5}
\sup_{C \in \mathcal{F}_2}\norm{\Phi(C)} = \max_{z \in {(\mathbb{D}^-)}^2}\abs{\Phi(z)}.
\ee
Unfortunately, when the operator theory community asked for the obvious analog of \eqref{eq1.5} to hold in dimension $3$, they were surprised to learn (\cite{var74},\cite{cradav}) that there are examples of $\Phi$ that are holomorphic on a neighborhood of ${(\mathbb{D}^-)}^3$ for which
\be\label{eq1.6}
\sup_{C \in \mathcal{F}_3}\norm{\Phi(C)} > \max_{z \in {(\mathbb{D}^-)}^3}\abs{\Phi(z)}.
\ee
As it was the \emph{right} side of \eqref{eq1.6}, defined as it is in concrete function theory terms, that was thought to be the object of interest, the \emph{left} side of \eqref{eq1.6} remained unexplored by operator theorists until the appearance in \cite{ag90} of
 the following enshrinement of von Neumann's Inequality as a definition. For $m\ge1$, let 
$$\fm=\set{C}{ C\ \text{ is an }m \text{-tuple of pairwise commuting contractions}}.$$
This is the collection defined by \eqref{eq1.0} in the case when $d=m$ and $\delta$ is the identity map on $\cm$.
\begin{defin}\label{def1.1}
For $m\ge1$ and $\Phi \in {\rm Hol}(\dm)$, define $\normm{\Phi}$ by
$$\normm{\Phi} = \sup_{\substack{C \in {\mathcal{F}}_m \\ \sigma(C) \subseteq \dm}} \norm{\Phi(C)}.$$
\end{defin}
The norm $\normm{}$ occurs in many areas of multivariable function theory and operator theory,
for example in Nevanlinna-Pick interpolation \cite{ag1,agmc_bid,baltre98}, in realization theory
\cite{bgr90,bsv05}, in the theory of matrix monotone functions \cite{amy12c}, 
Carath\'eodory-Julia theorems on the polydisk \cite{amy12a}, {\em etc.}.


We  can  now describe how to get bounds for the Oka representation. 
Notice from \eqref{eq1.0} that if $s<t$, then ${\mathcal{F}}_{t\delta} \subseteq {\mathcal{F}}_{s\delta}$. Equally obvious from \eqref{eq1.1} is that if ${\mathcal{F}}_\delta \subseteq {\mathcal{F}}_\gamma$, then ${\normpdel{\phi}} \le {\normpgam{\phi}}$. These two facts combine to show
 that $\normptdel{\phi}$ is a monotone decreasing function of $t$. Thus, we may define
$$\rho(\phi)=\lim_{t\to1^-}\normptdel{\phi}.$$
Theorem \ref{thm1.3} can be reformulated to assert that if $\phi$ is holomorphic on a neighborhood of $\kdel$, then there exists $\Phi$ holomorphic on  a neighborhood of ${(\mathbb{D}^-)}^{m}$ such that
$\phi =\Phi \circ \delta$ on $K_\delta$
if and only if $\rho(\phi) < \infty$. The following theorem describes the bounds we have for the classical Oka setting.

\begin{thm}\label{thm1.4}
Let $\delta$ be an $m$-tuple of polynomials in $d$ variables and let $\phi$ be holomorphic on a neighborhood of $\kdel$. If $\Phi$ is holomorphic on  a neighborhood of ${(\mathbb{D}^-)}^{m}$ and $\phi = \Phi\circ\delta$, then $\rho(\phi) \le \normm{\Phi}$. Furthermore, if $\epsilon > 0$, then there exists a $\Phi$ holomorphic on  a neighborhood of ${(\mathbb{D}^-)}^{m}$ such that $\phi = \Phi\circ\delta$ and $\normm{\Phi} < \rho(\phi)+\epsilon$.
\end{thm}

Note that Propositions~\ref{thm1.3} and \ref{thm1.2} are immediate corollaries of Theorem~\ref{thm1.4},
once one knows that every function $\Phi$  holomorphic on  a neighborhood of ${(\mathbb{D}^-)}^{m}$ has $\normm{\Phi}$ finite (Lemma~\ref{lem6.1}).

%
\subsection{$\hinfdel$ and $\hinfm$}\label{subsec1.4}

Up to now we have restricted ourselves to the classical Oka setting, in which one seeks to represent
functions  $\phi$ that are holomorphic on a neighborhood of $\kdel$. 
Sharper theorems are obtainable for functions defined only on $\gdel$. However, if $\phi$ is only defined on $\gdel$, then \eqref{eq1.1} doesn't make sense as it needn't be the case that $\phi(T)$ is well defined for all $T \in \fdel$. To accommodate this difficulty we modify the definition \eqref{eq1.1}  to sup only over those $T \in \fdel$ such that $\sigma(T) \subseteq \gdel$. Thus, for $\phi$ a holomorphic function on $\gdel$, we define $\normdel{\phi}$ by the formula,
\be\label{eq1.7}
{\norm{\phi}}_\delta = \sup_{\substack{T \in {\mathcal{F}}_\delta \\ \sigma(T) \subseteq G_\delta}} \norm{\phi(T)}.
\ee
Tautologically, we have 
$\normdel{\phi} \leq \normpdel{\phi} \leq \normtdel{\phi}$ in the case when $\phi$ is holomorphic on a neighborhood of $\kdel$ and $t < 1$ is sufficiently close to $1$.
 
Armed with this definition, we can define the space $\hinfdel$ to consist of all functions $\phi$ that are holomorphic on $\gdel$ and such that
$\normdel{\phi}$ is finite. Let us use $\bf m$ to denote the identity polynomial on ${\mathbb C}^m$. Then 
the norm $\normm{\Phi}$ from Definition \ref{def1.1}
is the same as the norm 
$\normbm{\Phi}$, and we 
 can define $\hinfm$ to consist of all functions $\Phi$ that are holomorphic on $\dm$ and such that
$\normm{\Phi}$  is finite.

 It turns out that $\hinfdel$ and $\hinfm$ equipped with these norms are Banach spaces. That these spaces are natural spaces in which to  study Oka representations is made clear by the following theorem.
\begin{thm}\label{thm1.5}
Let $\delta$ be an $m$-tuple of polynomials in $d$ variables and assume that $\phi$ is a holomorphic function on $\gdel$.
\begin{align*}
&a)\text{ There exists } \Phi \in \hinfm\text{ such that } \phi = \Phi \circ \delta \text{ if and only if } \phi \in \hinfdel.\\
&b)\text{ If } \Phi \in \hinfm \text{ and } \phi = \Phi \circ \delta, \text{ then } \normdel{\phi} \le \normm{\Phi}.\\
&c)\text{ If } \phi \in \hinfdel, \text{ there exists a } \Phi \in \hinfm \text{ such that } \phi = \Phi \circ \delta \text{ and } \normdel{\phi} = \normm{\Phi}.
\end{align*}
\end{thm}

\subsection{Realization formula}
\label{subsec1.5}

Our proofs rely on the existence of realizations.

\begin{defin}\label{def1.19}
Let $\phi$ be a function on $\gdel$. We say a $4$-tuple $(a,\beta,\gamma,D)$ is a realization for $\phi$ if $a \in \mathbb{C}$ and there exists a decomposed Hilbert space, $\m = \oplus_{l=1}^m \m_l$, such that the $2\times2$ matrix,
$$V=\begin{bmatrix}a&1\otimes\beta\\\gamma\otimes1&D\end{bmatrix},$$
acts isometrically on $\mathbb{C}\oplus\m$, $\delta(\lambda)$ acts on $\m$ via the formula,
$$\delta(\lambda)(\oplus_{l=1}^m x_l)=\oplus_{l=1}^m \delta_l(\lambda)x_l,$$
and
$$\phi(\lambda)=a+\langle \delta(\lambda){(1-D\delta(\lambda))}^{-1}\gamma,\beta \rangle
$$
for all $\lambda \in \gdel$.
\end{defin}

 C.-G.~Ambrozie and D.~Timotin \cite{at03} proved that a function
$\phi$ on $\gdel$ has a realization if and only if $\normmdel{\phi} \leq 1$,
where
\[
\normmdel{\phi} \ = \
\sup \{ \| \phi (T) \| \ : \
 \| \delta_l (T) \| < 1, \ 1 \leq l \leq m \} .
\]
The purpose of Section~\ref{sec3} is to develop the machinery to show
that
\[
\normmdel{\phi} \ = \
\normdel{\phi}, \qquad \forall \ \phi {\rm \ holomorphic\ on\ } \gdel 
\]
(Theorem~\ref{thm6.2}).
In fact, both norms agree with 
$\sup \{ \| \phi(T) \| \}$ as $T$ ranges over commuting $d$-tuples of
diagonalizable matrices in $\fdel$ (Theorem~\ref{thm8.1}).



\section {$H_\delta^\infty$}
Let $\delta =(\delta_1, \delta_2, ... ,\delta_m)$ be an $m$-tuple of nonconstant polynomials with complex coefficients in $d$ variables. We can think of $\delta$ as a map from $\mathbb{C}^d$ into $\mathbb{C}^m$ and define two sets in $\mathbb{C}^d$ by $G_{\delta}={\delta}^{-1}(\mathbb{D}^m)$ and $K_{\delta}={\delta}^{-1}((\mathbb{D}^-)^m)$.
\begin{lem}\label{lem2.1}
$G_{\delta}=K_{\delta}^\circ$
\end{lem}
\begin{proof}
Since $G_{\delta} \subseteq K_{\delta}$ and $G_{\delta}$ is open, $G_{\delta} \subseteq K_{\delta}^\circ$. If $\lambda \in K_{\delta}^\circ \setminus G_{\delta}$, then there exists an index $l$ with $\abs{\delta_l(\lambda)} = 1$. Since $\delta_l$ is assumed nonconstant, there exists a sequence $\lambda_n \to \lambda$ such that $\abs{\delta_l(\lambda_n)} > 1$. In particular,  $\lambda_n \notin K_\delta$, so, $\lambda \in \partial K_\delta$, a contradiction. This shows that $K_{\delta}^\circ \setminus G_{\delta}$ is empty. Hence, $G_{\delta}=K_{\delta}^\circ$.
\end{proof}
Note that even when $d=1$,  it need not be the case that $G_{\delta}^-$ coincides with $K_\delta$;
 and when $d>2$, $K_\delta$ need not be compact and $\gdel$ need not be bounded. 
In what follows $T$ will always denote a $d$-tuple of pairwise commuting bounded operators acting on a Hilbert space. Let $\fdel$ be as in (\ref{eq1.0}).
\begin{lem}\label{lem2.3}
If $T \in {\mathcal{F}}_\delta$ then $\sigma(T) \subseteq K_\delta$.
\end{lem}
\begin{proof}
Let $T\in{\mathcal{F}}_\delta$. Fix $l$. Since $\norm{{\delta}_l(T)} \le 1$, ${\sigma(\delta}_l(T)) \subseteq \mathbb{D}^-$. Hence, by the spectral mapping theorem, $\delta_l(\sigma(T)) \subseteq {\mathbb{D}}^-$. Thus, $\sigma(T) \subseteq {\delta}^{-1}((\mathbb{D}^-)^m) = K_{\delta}$.
\end{proof}

We now define an algebra of holomorphic functions on $G_\delta$. For $f \in {\rm Hol}(G_\delta)$ and $T$ with spectrum in $G_\delta$, $f(T)$ can be defined by the functional calculus. Let
\be\label{eq2.1}
{\norm{f}}_\delta = \sup_{\substack{T \in {\mathcal{F}}_\delta \\ \sigma(T) \subseteq G_\delta}} \norm{f(T)},
\ee
and define $H_\delta^\infty$ to consist of all $f \in {\rm Hol}(G_\delta)$ such that $\normdel{f}$ is finite.
\begin{prop}\label{prop2.1}
$H_\delta^\infty$ equipped with $\normdel{f}$ is a Banach algebra. Furthermore, if $H^\infty(G_\delta)$ denotes the space of bounded holomorphic functions on $G_\delta$ equipped with the sup norm, ${\norm{f}}_\infty$, then $H_\delta^\infty \subseteq H^\infty(G_\delta)$ and ${\norm{f}}_\infty \le {\norm{f}}_\delta$ for all $f \in H_\delta^\infty$.
\end{prop}
\begin{proof}
That $\hinfdel$ is a normed algebra is immediate from \eqref{eq2.1}. If $\lambda \in \gdel$, then $\lambda$ can be viewed as an element of $\fdel$ and, in addition, $\sigma(\lambda)=\{\lambda\} \subseteq \gdel$. Hence,
\be\label{eq2.2}
\normdel{f}=\sup_{\substack{T \in {\mathcal{F}}_\delta \\ \sigma(T) \subseteq G_\delta}} \norm{f(T)}
\ge \sup_{\lambda \in \gdel} \abs{f(\lambda)} = {\norm{f}}_\infty.
\ee
In particular, this implies $\hinfdel \subseteq \hinfofgd$.

Now assume that $\{f_n\}$ is a Cauchy sequence in $\hinfdel$. Since \eqref{eq2.2} implies that  $\{f_n\}$ is also Cauchy in $\hinfofgd$, there exists $f \in \hinfofgd$ such that $f_n \to f$ in $\hinfofgd$. Hence, by the continuity of the functional calculus,
\be\label{eq2.3}
\text{If } \sigma(T) \subseteq G_\delta, \text{ then } f_n(T) \to f(T).
\ee
Now, since $\{f_n\}$ is Cauchy in $\hinfdel$, there exists $M$ such that $\norm{f_n} \le M$ for all $N$. Hence, if $T \in \fdel$ and $\sigma(T) \subseteq G$, then \eqref{eq2.3} implies that
$$\norm{f(T)} = \lim_{n \to \infty} \norm{f_n(T)} \le M.$$
Hence, since $\norm{f(T)} \le M$ whenever $T \in \fdel$ and $\sigma(T) \subseteq G_\delta$, \eqref{eq2.1} implies that $\normdel{f} \le M$, and we see that $f \in \hinfdel$.

To see that $f_n \to f$ in $\hinfdel$, fix $\epsilon > 0$. Choose $N$ so that $m,n \ge N \Rightarrow \normdel{f_n-f_m} < \epsilon$. If $T \in \fdel$ and $\sigma(T) \subseteq G_\delta$, then
$$\norm{f_n(T)-f_m(T)} < \epsilon.$$
Thus, letting $m\to\infty$, we see that if $T \in \fdel$ and $\sigma(T) \subseteq G_\delta$, then
$$n \ge N \Rightarrow \norm{f_n(T)-f(T)} \le \epsilon.$$
But then it follows from \eqref{eq2.1} that
$$n \ge N \Rightarrow \normdel{f_n-f} \le \epsilon.$$
Thus, $f_n \to f$ in $\hinfdel$.
\end{proof}

We close this section with the following proposition which identifies in operator theory terms when the space $\hinfdel$ contains the functions that are holomorphic on a neighborhood of $\kdel$. Let us define
\[
\fdel^0 \ = \
\{ T \in \fdel : \sigma(T) \subset \gdel \}.
\]
\begin{prop}\label{prop2.2}
The following are equivalent.
\begin{align*}
&\text{a)} \quad \phi \in \hinfdel \text{ whenever } \phi  \text{ is holomorphic on a neighborhood of } \kdel.\\
&\text{b)} \quad \lambda^r \in \hinfdel \text{ for } r=1,\ldots,d.\\
&\text{c)} \quad \fdel^0 \text{ is bounded}.
\end{align*}
\end{prop}
\begin{proof}
As $\lambda^r$ is holomorphic on a neighborhood of $\kdel$, a) implies b). That b) implies c) follows immediately from \eqref{eq2.1}. Suppose that c) holds. If $\phi$ is holomorphic on a neighborhood of $\kdel$ yet $\phi \notin \hinfdel$, then there exist $T_n \in \fdel^0$ such that
\be \label{eq2.4}
\norm{\phi(T_n)} \to \infty.
\ee
As $\fdel^0$ is bounded, $T = \oplus{T_n} \in \fdel$. Hence by Lemma \ref{lem2.3}, $\sigma(T) \subseteq \kdel$ and $\phi(T)$ is a well defined operator. But then
$$\norm{\phi(T_n)} \le \norm{\oplus{\phi(T_n)}} = \norm{\phi(\oplus{T_n})} = \norm{\phi(T)}$$
contradicting \eqref{eq2.4}.
\end{proof}


\section{Hereditary Calculus}
\label{sec3}

For $G$ an open set in ${\mathbb{C}}^d$, we let ${\mathcal{H}}(G)$ denote the collection of functions, $h=h(\lambda,\mu)$, defined for $(\lambda,\mu) \in G \times G$, such that $h$ is holomorphic in $\lambda$ on $G$ for each fixed $\mu \in G$ and $h$ is anti-holomorphic 
(i.e. $\overline{h}$ is holomorphic) in $\mu$ on $G$ for each fixed $\lambda \in G$. If we equip ${\mathcal{H}}(G)$ with the topology of uniform convergence on compact subsets of $G \times G$, then ${\mathcal{H}}(G)$ is a locally convex topological vector space with a topology induced by a complete translation invariant metric. Furthermore
${\mathcal{H}}(G)$ is isomorphic as a topological vector space with $\overline{{\rm Hol}(G)} \otimes {\rm Hol}(G)$, the completion of the projective tensor product,  via the
 continuous linear extension to $\overline{{\rm Hol}(G)} \otimes {\rm Hol}(G)$ of the bilinear map defined by
$$ \overline{{\rm Hol}(G)} \otimes {\rm Hol}(G) \owns \overline{g(\mu)} \otimes f(\lambda) \mapsto \overline{g(\mu)}f(\lambda) \in {\mathcal{H}}(G).$$
(See \cite[Thm. 51.6]{tre}.

In particular, if $B$ is a Banach space and $u:\overline{{\rm Hol}(G)} \times {\rm Hol}(G) \to B$ is a jointly continuous $B$-valued bilinear map, then there exists a continuous linear map $\Gamma:{\mathcal{H}}(G) \to B$ such that
\be\label{eq3.1}
u(\overline{g(\mu)},f(\lambda)) = \Gamma(\overline{g(\mu)}f(\lambda))
\ee
for all $f,g \in {\rm Hol}(G)$ (see \cite[p. 325]{Jar81}).

If $\mathcal{H}$ is a Hilbert space, we let $\lofh$ denote the $C$*-algebra of bounded operators on $\mathcal{H}$. If $T = (T^1,\ldots,T^d)$ is a $d$-tuple of pairwise commuting elements of $\mathcal{L}(\mathcal{H})$ and $\sigma(T) \subseteq G$, then by the continuity of the functional calculus, $u$, defined by $u(\overline{g},f) = {g(T)}^*f(T)$, is a jointly continuous $\lofh$-valued bilinear map on $\overline{{\rm Hol}(G)} \times {\rm Hol}(G)$. Hence, if $\Gamma:{\mathcal{H}}(G) \to \lofh$ is defined by \eqref{eq3.1}, we can define the \emph{hereditary calculus for $T$} by setting $h(T)=\Gamma(h)$ for all $h \in \h$. Note that with this definition, we have that
\be \label{eq3.2}
\big{[}\overline{g(\mu)}h(\lambda,\mu)f(\lambda)\big{]}(T)={g(T)}^*h(T)f(T)
\ee
for all $f,g \in {\rm Hol}(G)$ and all $h \in \hofg$.

For $a \in \hofg$ we define $a^* \in \hofg$ by $a^*(\lambda,\mu)=\overline{a(\mu,\lambda)}$. Note that with this notation, \eqref{eq3.2} takes on the more pleasing form,
\be \label{eq3.2.1}
(g^*hf)(T)={g(T)}^*h(T)f(T).
\ee
We define $\rofg=\{a \in \hofg | a=a^*\}$ and observe that $\rofg$ is a real locally convex space with the induced topology from $\hofg$. Also, as $h^*(T)={h(T)}^*$ whenever $\sigma(T) \subseteq G$ and $h \in \hofg$, we see that if $\sigma(T) \subseteq G$ and $a \in \rofg$, then $a(T)$
 is self adjoint. We say that $a \in \hofg$ is \emph{positive semidefinite}, and write $a \ge 0$, if
\be \label{eq3.3}
\sum\limits_{i,j=1}^n a(\lambda_j,\lambda_i)c_j\overline{c_i} \ge 0,
\ee
whenever $n$ is a positive integer, $\lambda_1, \ldots, \lambda_n \in G$, and $c_1, \ldots, c_n \in \mathbb{C}.$
We set $\pofg=\set{a \in \hofg}{a \ge 0}$. It follows easily from \eqref{eq3.3} that $\pofg$ is a closed cone in $\rofg$.
\begin{prop}\label{prop3.1}
If $a \in \hofg$, then $a \in \pofg$ if and only if there exist a Hilbert space $\mathcal{M}$ and a holomorphic map $u:G \to \mathcal{M}$ such that
\be\label{eq2.4x}
a(\lambda,\mu)={ \langle u(\lambda),u(\mu) \rangle}_{\mathcal{M}}
\ee
for all $\lambda,\mu \in G$.
\end{prop}
\begin{proof}
By Aronszjan's construction \cite{aro50,ampi} there is a Hilbert space
$\mathcal{M}$ of functions on $G$, with reproducing kernel $k$, such that
\be\label{eq3.5}
a(\lambda,\mu)={\langle k_{\lambda},k_{\mu} \rangle}_{\mathcal{M}} \qquad \lambda,\mu \in G.
\ee
Since $a$ is anti-holomorphic in $\mu$ for each fixed $\lambda \in G$,
\eqref{eq3.5} implies that if $f \in {\rm span}\set{k_\lambda}{\lambda \in G}$, 
then $f$ is anti-holomorphic on $G$. Since ${\rm span}\set{k_\lambda}{\lambda \in G}$ is dense in $\mathcal{M}$ and ${\norm{k_\mu}}^2=a(\mu,\mu)$ is bounded on compact subsets of 
$G$, in fact $f$ is anti-holomorphic on $G$ for all $f \in \mathcal{M}$. Hence, if we define $u(\lambda)=k_{\lambda}$, then $\langle u(\lambda),f \rangle=\overline{f(\lambda)}$ is holomorphic for all $f \in \mathcal{M}$, and we see that $u$ is weakly holomorphic on $G$. As $u$ is weakly holomorphic, $u$ is holomorphic. \eqref{eq2.4x} follows from \eqref{eq3.5}.
\end{proof}
\begin{lem}\label{sumofdyads}
If $a \in \pofg$, then there exists a countable sequence $\{f_i\}$ in ${\rm Hol}(G)$ such that $$a(\lambda,\mu)=\sum_i\overline{f_i(\mu)}f_i(\lambda).$$
\end{lem}
\begin{proof}
By Proposition \ref{prop3.1} there there exist a Hilbert space $\mathcal{M}$ and a holomorphic map $u:G \to \mathcal{M}$ such that \eqref{eq2.4} holds. As $u$ is holomorphic, ${\mathcal{M}}_0$, the closed linear span of $\set{u(\lambda)}{\lambda \in G}$ in $\mathcal{M}$ is separable. Let $\{e_i\}$ be a countable basis for $\mathcal{M}_0$ and define $f_i$ by
 $f_i={\langle u(\lambda),e_i \rangle }_\mathcal{M}$.
\end{proof}
\begin{lem}\label{haoftpos}
Let $h \in \rofg$, $a \in \pofg$, and assume that $T$ is a $d$-tuple of pairwise commuting operators with $\sigma(T) \subseteq G$. If $h(T) \ge 0$, then $(ha)(T) \ge 0$.
\end{lem}
\begin{proof}
By Lemma \ref{sumofdyads} there is a sequence $\{f_i\}$ in ${\rm Hol}(G)$ such that $a=\sum_i{f_i}^*f_i$. Hence, if $h(T) \ge 0$,
$$(ha)(T) = \sum_i(f_i^*hf)(T) = \sum_i{f_i(T)}^*h(T)f(T) \ge0.$$
\end{proof}

\begin{defin}
We say that $\c \subseteq \rofg $ is a {\bf hereditary cone}
 on $G$ if $\c$ is a cone in $\rofg$ with the property that $ha \in \c$ whenever $h \in \c$ and $a \in \pofg$.
\end{defin}

One way to construct hereditary cones is to let $\Omega$ be a subset of $\rofg$ and to define $\langle \Omega \rangle$ by
$$\langle \Omega \rangle=\set{\sum_{i=1}^{n} h_ia_i}{n \in \mathbb{N}, h_1, \ldots,h_n \in \Omega, a_1, \ldots,a_n \in \pofg}.$$
Evidently, $\langle \Omega \rangle$ is the hereditary cone generated by $\Omega$, i.e. the smallest hereditary cone $\c \subseteq \rofg$ such that $\c \supseteq \Omega$. 
\begin{defin}\label{def4.1}
For $\f$ a collection of pairwise commuting operator $d$-tuples and $G$ an open set in $\cd$, we define $\fperp(G) \subseteq \rofg$ by
$$\fperp(G) = \set{h \in \rofg}{h(T) \ge 0 \text{ whenever }T \in \f \text{ and } \sigma(T) \subseteq G}.$$
\end{defin}
\begin{lem}\label{lem4.3}
If $\f$ is a collection of operators and $G$ is an open set in $\cd$, then $\fperp(G)$ is a hereditary cone on $G$.
\end{lem}
\begin{proof}
Let $h \in \fperp(G)$ and $a \in \pofg$. If $T \in \f$ and $\sigma(T) \subseteq G$, then $h(T) \ge 0$. Hence, by Lemma \ref{haoftpos}, $(ha)(T) \ge 0$. Thus, $ha \in \fperp(G)$.
\end{proof}
\section{The Realization Formula}
\label{sec6}



We record the following two simple lemmas for future use. We choose to deduce them as corollaries of  Proposition \ref{prop2.2}. Alternative 
direct, constructive proofs of them are obtainable based either on the theory of power series or iterated Cauchy-Riesz-Dunford integrals.
\begin{lem}\label{lem6.1}
If $\Phi$ is holomorphic on a neighborhood of $\cldm$, then $\Phi \in \hinfm$.
\end{lem}
\begin{proof}
The lemma is an immediate consequence of b) implies a) in Proposition \ref{prop2.2}.
\end{proof}
\begin{lem}\label{lem6.2}
If $s > 1$, then $H^\infty(s\dm) \subseteq \hinfm$. Furthermore, there exists a constant $c$
depending on $s$ such that
$$\normm{\Phi} \le c\sup_{\lambda \in \dm}\abs{\Phi(\lambda)}$$
for all $\Phi \in H^\infty(s\dm)$.
\end{lem}
If $\mathcal{B}$ is a Banach space, we let $ball(\mathcal{B})$ denote the closed unit ball of $\mathcal{B}$.
\begin{defin}\label{def6.1}
Let $\phi$ be a function on $\dm$. We say a $4$-tuple $(a,\beta,\gamma,D)$ is a realization for $\phi$ if $a \in \mathbb{C}$ and there exists a decomposed Hilbert space, $\m = \oplus_{l=1}^m \m_l$, such that the $2\times2$ matrix,
$$V=\begin{bmatrix}a&1\otimes\beta\\\gamma\otimes1&D\end{bmatrix},$$
acts isometrically on $\mathbb{C}\oplus\m$, $z$ acts on $\m$ via the formula,
$$z(\oplus_{l=1}^m x_l)=\oplus_{l=1}^m z_lx_l,$$
and
$$\phi(z)=a+\langle z{(1-Dz)}^{-1}\gamma,\beta \rangle$$
for all $z \in \dm$.
\end{defin}

The following theorem was proved in \cite{ag90}.
\begin{thm}\label{thm6.1}
Let $\phi$ be a function defined on $\dm$. The following are equivalent.
\begin{align*}
&a) \qquad \phi \in ball(\hinfm)\\
&b) \qquad 1-\phi^*\phi \in {\mathcal F}_m^\perp\\
&c) \qquad \phi \text{ has a realization.}
\end{align*}
\end{thm}
 Definition \ref{def6.1} and Theorem \ref{thm6.1} 
can be extended
to the $\hinfdel$ setting. Recall Definition~\ref{def1.19}:

{\em
Let $\phi$ be a function on $\gdel$. We say a $4$-tuple $(a,\beta,\gamma,D)$ is a realization for $\phi$ if $a \in \mathbb{C}$ and there exists a decomposed Hilbert space, $\m = \oplus_{l=1}^m \m_l$, such that the $2\times2$ matrix,
$$V=\begin{bmatrix}a&1\otimes\beta\\\gamma\otimes1&D\end{bmatrix},$$
acts isometrically on $\mathbb{C}\oplus\m$, $\delta(\lambda)$ acts on $\m$ via the formula,
$$\delta(\lambda)(\oplus_{l=1}^m x_l)=\oplus_{l=1}^m \delta_l(\lambda)x_l,$$
and
$$\phi(\lambda)=a+\langle \delta(\lambda){(1-D\delta(\lambda))}^{-1}\gamma,\beta \rangle$$
for all $\lambda \in \gdel$.
}

We adopt the notation $\cdel$ for the hereditary cone in $\mathcal{R}(\gdel)$ generated by the elements $1-{\delta_1}^*\delta_1, \ldots,1-{\delta_m}^*\delta_m$.
\begin{thm}\label{thm6.2}
Let $\phi$ be a function defined on $\gdel$. The following are equivalent.
\begin{align*}
&a) \qquad \phi \in ball(\hinfdel);\\
&b) \qquad \normmdel{\phi} \leq 1 ; \\
&c) \qquad 1-\phi^*\phi \in \cdel;\\
&d) \qquad \phi \text{ has a realization.}
\end{align*}
\end{thm}
\begin{proof}
The equivalence of {\em (b), (c)} and {\em (d)} is a special case of Theorem 3 in
the paper \cite{at03} by Ambrozie and Timotin; see also \cite{babo04}.

By Lemma 1 of \cite{at03}, if 
\begin{equation}
\label{eq737}
\| \delta_l (T) \| < 1, \ 1 \leq l \leq m,
\end{equation}
then
$\sigma(T) \subset \gdel$, so tuples satisfying (\ref{eq737}) lie in $\fdel$, and hence
$\normmdel{\phi} \leq \normdel{\phi}$. This means $\fdel^\perp (\gdel)
\subseteq \cdel$. The other inclusion follows from Lemma~\ref{lem4.3}.
\end{proof}

\section{Oka Mappings}
\label{sec7}

In this section we shall prove the theorems described in the introduction. Theorems \ref{thm1.2}, \ref{thm1.3}, and \ref{thm1.4} will be deduced from Theorem \ref{thm1.5}.
\subsection{The proof of Theorem \ref{thm1.5}}
First suppose $\phi \in ball(\hinfm)$. By Theorem \ref{thm6.2}, $\phi$ has a realization, $(a,\beta,\gamma,D)$, such that
\be \label{eq7.1}
\phi(\lambda)=a+{\langle \delta(\lambda){(1-D\delta(\lambda))}^{-1}\gamma,\beta \rangle }_\m
\ee
for all $\lambda \in \gdel$. It follows that if we define $\Phi$ on $\dm$ by
\be \label{eq7.2}
\Phi(z)=a+{\langle z{(1-Dz))}^{-1}\gamma,\beta \rangle }_\m,
\ee
then $\Phi \in ball(\hinfm)$ (Theorem \ref{thm6.1}) and $\phi(\lambda)=\Phi(\delta(\lambda))$ for all $\lambda \in \gdel$. This proves parts a) and c) of Theorem \ref{thm1.5}.

To prove part b), assume that $\Phi \in \hinfm$ and $\normm{\Phi} = 1$. Define a function $\phi$ on $\gdel$ by the formula, $\phi(\lambda)=\Phi(\delta(\lambda))$. By Theorem \ref{thm6.1}, there exists a realization $(a,\beta,\gamma,D)$ for $\Phi$ such that \eqref{eq7.2} holds for all $z \in \dm$. Hence \eqref{eq7.1} holds for all $\lambda \in \gdel$ and we see via Theorem \ref{thm6.2} that $\phi \in \hinfdel$ and $\normdel{\phi} \le 1$. Hence, $\normdel{\phi} \le 1 = \normm{\Phi}$. This proves b) and completes the proof of  Theorem \ref{thm1.5}.

\subsection{The Proof of Theorem \ref{thm1.4}}
Let $\delta$ be an $m$-tuple of polynomials in $d$ variables and assume that $\phi$ is holomorphic on a neighborhood of $\kdel$.

First assume that $\Phi$ is holomorphic on a neighborhood of $\cldm$ and $\phi = \Phi\circ\delta$. Fix $\epsilon > 0$. Using Lemma \ref{lem6.2} choose $t <1$, sufficiently close to $1$, so that $\normm{\Phi(\frac{z}{t})} < \normm{\Phi} + \epsilon$. Define $\Psi$ by setting $\Psi(z)=\Phi(\frac{z}{t})$. Evidently, as $\Psi \in \hinfm$ and $\phi(\lambda) = \Phi(\delta(\lambda)) = \Psi(t\delta(\lambda))$, by part b) of Theorem \ref{thm1.5}, $\normtdel{\phi} \le \normm{\Psi}$. Thus,
$$\rho(\phi) \le \normtdel{\phi} \le \normm{\Psi} \le \normm{\Phi}+\epsilon.$$
As $\epsilon$ is arbitrary, this proves the first assertion made in Theorem \ref{thm1.4}.

To prove the second assertion of the theorem, first fix $\epsilon > 0$. As
$$\lim_{t \to 1^-}\normtdel{\phi} = \rho(\phi),$$
there exists $t<1$ such that $\normtdel{\phi} < \rho(\phi) + \epsilon$. By part c) of Theorem \ref{thm1.5}, there exists $\Psi \in \hinfm$ such that $\phi(\lambda) = \Psi(t\delta(\lambda))$ and $\normm{\Psi} = \normtdel{\phi}$. Finally, note that if $\Phi$ is defined by $\Phi(z) = \Psi(tz)$ then as $t<1$, $\normm{\Phi}<\normm{\Psi}$. With these constructions we have that $\phi = \Phi \circ \delta$ and
$$\normm{\Phi} \le \normm{\Psi} = \normtdel{\phi} < \rho(\phi) + \epsilon.$$
This completes the proof of Theorem \ref{thm1.4}.

\section{Remarks}
\label{sec8}

It is worth noting that the norm $\normdel{\phi}$ is always achieved by
taking the supremum of $\phi(T)$ 
as $T$ ranges over tuples of simultaneously diagonalizable matrices in $\fdel$.
Indeed, in \cite{lewe92}, it was asked whether (\ref{eq1.6})  could hold for some 
generic $C$ (generic means all the eigenvalues are distinct).
It was shown that it could in \cite{lot94} and \cite{lotste}.
The existence of such a $C$ would also follow from the non-generic examples in
\cite{cradav} or \cite{var74} and the following theorem.

\begin{thm}
\label{thm8.1}
Let $\phi$ be a function defined on $\gdel$. Then
\begin{equation}
\label{eq88}
\normdel{\phi} \ = \
\sup \{ \| \phi(T) \| \ : \ T\ {\rm is\ a\ } d{\rm -tuple\ of\ 
generic\ matrices\ in\ } \fdel \}.
\end{equation}
\end{thm}
\begin{proof}
The inequality $\geq$ is obvious. Assume the right-hand side of (\ref{eq88}) is $1$.
As commuting {\em diagonalizable} matrices can be perturbed to commuting generic matrices
(this need not be true for non-diagonizable matrices \cite{gu92}), then 
\[
\sup \{ \| \phi(T) \|  :  T\ {\rm is\ a\ } d{\rm -tuple\ of\ 
commuting\ diagonalizable\ 
 matrices\ in\ } \fdel \}
\]
is also $1$.
If $T$ is a commuting diagonizable $d$-tuple of $n$-by-$n$ matrices, we can choose common
eigenvectors $v_1, \dots , v_n$ so that
\[
T^r v_j \ = \
\l^r_j v_j, \qquad 1 \leq r \leq d, \ 1 \leq j \leq n.
\]
Let 
$K$ be the Gram matrix $K_{ij} = \langle v_j, v_i \rangle$. The assertion
$\| \delta_l(T) \| \leq 1$ is the same as
\begin{equation}
\label{eq82}
\left[ (1 - \overline{\delta_l(\l_i)} \delta_l (\l_j) ) K_{ij} \right]
\ \geq \ 0 .
\end{equation}
Thus we have: whenever $\l_1, \dots, \l_n$ is a finite set in $\gdel$,
and $K$ is a positive definite matrix such that (\ref{eq82}) holds for
$1 \leq l \leq m$, then 
\[
\left[ (1 - \overline{\phi(\l_i)} \phi (\l_j) ) K_{ij} \right]
\ \geq \ 0 .
\] 
By the usual Hahn-Banach argument (see \cite[Sec. 11.1]{ampi}), this proves that 
$1 - \phi^* \phi$ is in $\cdel$, and hence $\normdel{\phi} \leq 1$.
\end{proof}


%

\bibliography{../references}.
\end{document}